\documentclass[12pt,letterpaper]{amsart}
\usepackage[makeindex]{imakeidx}
\usepackage{ amsmath,amsthm, amsbsy,amsfonts,amssymb, txfonts}
\usepackage[normalem]{ulem}
\usepackage{stmaryrd,mathrsfs,graphicx, supertabular, amscd, tikz-cd, tikz}
\tikzcdset{row sep/my size/.initial=6mm}
\tikzcdset{column sep/my size/.initial=6mm}
\usepackage{stackrel}
\usetikzlibrary{matrix,arrows,decorations.pathmorphing}

\usepackage{ulem}

\usepackage[active]{srcltx}
\allowdisplaybreaks
\numberwithin{equation}{section}

\usepackage[
	hypertexnames=false,
	hyperindex,
	pagebackref,
	breaklinks=true,
	bookmarks=false,
	colorlinks,
	linkcolor=blue,
	citecolor=red,
	urlcolor=blue,
]{hyperref}
\usepackage{hyperref}

\def\QQ{{\mathbb Q}}

\def\ZZ{{\mathbb Z}}

\def\g{\gamma}

\def\Bcal{{\mathcal B}}
\def\Ccal{{\mathcal C}}

\def\Ical{{\mathcal I}}

\def\Pcal{{\mathcal P}}

\def\mfrak{{\mathfrak{m}}}

\def\vfrak{\mathfrak{v}}
\def\wfrak{\mathfrak{w}}
\def\ufrak{\mathfrak{u}}

\def\Ccalk{{\Ccal_k}}

\newcommand\gr{\operatorname{gr}}

\newcommand\Tcal{\operatorname{{\mathcal T}}}

\newcommand\Sp{\operatorname{Sp}}

\DeclareMathOperator{\id}{id}
\def\f{\phi}

\newtheorem{atheorem}{Theorem}

\newtheorem{theorem}{Theorem}[section]
\newtheorem{lemma}[theorem]{Lemma}
\newtheorem{proposition}[theorem]{Proposition}

\newtheorem{definition}[theorem]{Definition}

\theoremstyle{remark}
\newtheorem{example}[theorem]{Example}

\newtheorem{remark}[theorem]{Remark}

\title[$\Gamma_{g,1}$-invariants of Passi representations]{The Mapping class group invariants of the truncated group ring}
\author{Andreas Stavrou}
\address{Mathematics Department, University of Chicago}
\email{andreasstavrou@uchicago.edu}

\begin{document}
\begin{abstract} We compute the invariant subspace of the rational group ring of a surface, truncated by powers of the augmentation ideal, under the action of the mapping class group. The surface is compact, oriented with one boundary component. This provides the first group cohomology computation for the mapping class group with non-symplectic coefficients since Kawazumi--Souli\'e. Our computation is valid in a range growing with the genus.
\end{abstract}

\maketitle

% add in Bibliography Mihailovs

\section{Introduction}
Let $\Sigma_{g,1}$ be a compact oriented genus $g$ surface with one boundary component, and  $\Gamma_{g,1}$ denote its mapping class group. There is a natural action of $\Gamma_{g,1}$ on the fundamental group $\pi=\pi_1(\Sigma_{g,1},p)$ (relative to a point on the boundary) which extends to the group ring $\QQ\pi$. This action preserves the filtration of $\QQ\pi$ by powers of the augmentation ideal $\Ical$, the kernel of the augmentation homomorphism $\varepsilon: \QQ\pi\to \QQ$, sending every $\g\in \pi$ to $1$. In this paper we focus on the $\Gamma_{g,1}$-representations $\Pcal_k=\QQ\pi/\Ical^{k+1},$ sometimes referred to as the \textit{Passi representations}, and our main result determines their $\Gamma_{g,1}$-invariants $[\Pcal_k]^{\Gamma_{g,1}}$. 

The representations $\Pcal_k$ play a key role in the study of configuration spaces, such as in the work of Moriyama, Looijenga, and the author \cite{Moriyama,looijenga2024motivic,LooijengaStavrou25}, interpolating in complexity between the full group ring $\QQ\pi$ and the first homology $H=H_1(\Sigma_{g,1};\QQ)$, on which $\Gamma_{g,1}$ acts through the symplectic group $\Sp_{2g}(\ZZ)$. As the invariants $[\Pcal_k]^{\Gamma_{g,1}}$ coincide with the group cohomology $H^0(\Gamma_{g,1};\Pcal_k)$, our main result is the first cohomological computation of $\Gamma_{g,1}$ with non-symplectic coefficients since Kawazumi-Souli\'e \cite{KawazumiSoulie24}. We will pursue analogous higher degree computation in joint work with the last author \cite{SoulieStavrou2025}.

What do we expect to find in $[\Pcal_k]^{\Gamma_{g,1}}$? On the one hand, the boundary loop $\zeta\in \pi$ is fixed by $\Gamma_{g,1}$, so polynomials in $\zeta$ and $\zeta^{-1}$ give $\Gamma_{g,1}$-invariant elements of $\QQ\pi$. (In Theorem \ref{thm:invariantsoffullring}, we show this is all of $[\QQ\pi]^{\Gamma_{g,1}}$).  On the other hand, $\Pcal_k$ contains the submodule $\Ical^k/\Ical^{k+1}$ which, by a classical result of Fox \cite{Fox}, is naturally isomorphic to $H^{\otimes k}$. The $\Gamma_{g,1}$-invariants of the latter are the classically known $\Sp_{2g}(\ZZ)$-invariants $[H^{\otimes k}]^{\Sp_{2g}(\ZZ)}$ (described in Section \ref{sec:chorddiagrams}), and also lie in $[\Pcal_k]^{\Gamma_{g,1}}$. Our main result shows that, in a stable range, these account for all of $[\Pcal_k]^{\Gamma_{g,1}}$.

\begin{atheorem}\label{thm:main}
    If $k+1\le 2g$, then the $\Gamma_{g,1}$-invariant part of $\Pcal_k$ is the direct sum $$\ZZ\langle (\zeta-1)^i+\Ical^{k+1}:2i< k\rangle\oplus [H^{\otimes k}]^{\Sp_{2g}(\ZZ)}.$$
\end{atheorem}

In other words, passing from $k$ to $k+1$ via the surjections $\Pcal_{k+1}\twoheadrightarrow \Pcal_k$, gives maps $[\Pcal_{k+1}]^{\Gamma_{g,1}}\twoheadrightarrow [\Pcal_k]^{\Gamma_{g,1}}$ that surject on the $\zeta$-part but do not interact with the $[H^{\otimes k}]^{\Sp_{2g}(\ZZ)}$-parts. This is in no contradiction with the fact that the invariant functor is only left-exact and so only preserves injections.

While the condition $k+1\le 2g$ may be an artefact of the proof, it aligns with the best known stable ranges of group cohomology with polynomial coefficients of degree $k$ proved by Boldsen \cite{Boldsen}. In Remark \ref{rem:unstablerange}, we suggest a method to enlarge this range.

\subsection{Outline}
Our proof will be an induction using the recursive nature of the $\Pcal_k$. In Section \ref{sec:reductiontomainlemma}, we reduce the theorem to computing the $\Gamma_{g,1}$-invariants of the quotient $\Ical^{k}/\Ical^{k+2}$, which is an extension of $H^{\otimes k}$ by $H^{\otimes k+1}$ related to the Johnson-Morita crossed homomorphism from \cite{Morita93CrossedHom}. These invariants are computed in Section \ref{sec:proofofmainlemma} via a combinatorial analysis of monomials in a symplectic basis of $H$ that appear in the ``chord-diagram'' description of $[H^{\otimes k}]^{\Sp_{2g}(\ZZ)}]$.

\subsection{Acknowledgments} I would like to thank Arthur Souli\'e and Louis Hainaut for reading earlier drafts and the inspiring conversations.

\section{Reduction to the main lemma}\label{sec:reductiontomainlemma}
\subsection{Invariant theory and chord diagrams}\label{sec:chorddiagrams}
Let us fix a symplectic basis $a_{1}, b_1 \ldots, a_{g},b_g$ for $H$ so that the intersection pairing evaluates as $\langle a_i, b_{j}\rangle=\delta_{i,j}$ for $1\le i,j\le g$, and $\langle a_i,a_j\rangle=\langle b_i,b_j\rangle=0$.
Then the pairing $\langle -,-\rangle$ produces by self-duality an invariant element $\omega\in H^{\otimes 2}$ given in this basis by  $$\omega=\sum_{i=1}^g a_i\otimes b_i-b_i\otimes a_i,$$
or, in shorthand, by $\sum_{i=1}^g a_i\wedge b_i$, where we interpret $x\wedge y=x\otimes y-y\otimes x$.

All invariant elements of $H^{\otimes 2l}$ for $l\ge 0$ are generated by $\omega$ as follows.
A \textit{chord diagrams of size $2l$} is an ordered partition of the set $\{1,\ldots, 2l\}$ into $k$ pairs $((p_1,q_1),...,(p_l,q_l))$ such that $p_i<q_i$ for all $i=1,...,l$ and $p_1<p_2<\cdots<p_l$. Let $\Ccal_{2l}$ be the set of all chord diagrams. For each $C\in \Ccal_{2l}$, there is an associated invariant
$$\omega_C=\otimes_{i=1}^l\omega_{p_i,q_i}\in H^{\otimes 2l},$$
by inserting $\omega$ in the tensor slots $(p_i,q_i)$ for each $i$. For example, the \textit{trivial} chord diagram $C_0=((1,2),(3,4),\ldots, (2l-1,2l))$ corresponds to the tensor power $\omega^{\otimes l}$. We thus have a map $\Omega_{2l}:\QQ\Ccal_l\to [H^{\otimes 2l}]^{\Sp_{2g}(\ZZ)}$. The following can be deduced from Section 11.6.3 of \cite{Procesi}. 
\begin{theorem}\label{thm:chorddiagramsinvariants}
    For $k$ odd, $[H^{\otimes k}]^{\Sp_{2g}(\ZZ)}$ is trivial. For  $k$ even, the map $\Omega_k$ is surjective, and is, furthermore, an isomorphism if $k\le 2g$.
\end{theorem}

\subsection{On the associated graded}
By a theorem of Fox \cite{Fox}, the augmentation filtration $\QQ\pi\supset \Ical\supset \Ical^2\supset\cdots$ has associated graded ring $\gr^\Ical_*\QQ\pi=\oplus_{k\ge 0}\Ical^{k}/\Ical^{k+1}$ isomorphic to the free tensor algebra $T[H]$ on $H$. The correspondence is given by 
\begin{equation*}\label{eq:correspondenceassgraded}
    (\g_1-1)\cdots(\g_k-1)+\Ical^{k+1}\in \Ical^k/\Ical^{k+1}\longleftrightarrow [\g_1]\otimes \cdots \otimes [\g_k]\in H^{\otimes k},
\end{equation*}
where $\g_1,\ldots, \g_k\in \pi$, and $[\g]\in H$ denotes the abelianisation of $\g\in\pi$. 
\begin{example}\label{ex:zetaISomega}
    The element $\zeta$ is in the commutator subgroup of $\pi$ and so $\zeta-1\in \Ical^2$. It is a standard computation that $\zeta-1+\Ical^3\longleftrightarrow \omega\in H^{\otimes 2}$. It follows, more generally, that each $(\zeta-1)^i+\Ical^{2i+1}$ corresponds to $\omega^{\otimes i}\in H^{\otimes 2i}$.
\end{example}

It follows that the longer quotient $\Ical^k/\Ical^{k+2}$ sits in a short exact sequence
\begin{equation}\label{eq:2stepSES}
    0\to H^{\otimes k+1}\to \Ical^k/\Ical^{k+2}\to H^{\otimes k}\to 0
\end{equation}
of $\Gamma_{g,1}$-representations, for each $k\ge 0$, where the extremal terms are symplectic. The left exactness of the invariants functor gives the exact sequence
\begin{equation*}\label{eq:2stepSESinvariants}
    0\to [H^{\otimes k+1}]^{\Sp_{2g}(\ZZ)}\to [\Ical^k/\Ical^{k+2}]^{\Gamma_{g,1}}\to [H^{\otimes k}]^{\Sp_{2g}(\ZZ)},
\end{equation*}
which in the case $k$ is even, say $k=2l$, reduces by Theorem \ref{thm:chorddiagramsinvariants} to an injection
\begin{equation}\label{eq:invariantinjection}
    [\Ical^{2l}/\Ical^{2l+2}]^{\Gamma_{g,1}}\hookrightarrow [H^{\otimes 2l}]^{\Sp_{2g}(\ZZ)}.
\end{equation}
\begin{lemma}[Main lemma]\label{lem:main}
    If $2l+1\le 2g$, the image of \eqref{eq:invariantinjection} is spanned by $\omega^{\otimes l}$ and its domain by $(\zeta-1)^l+\Ical^{2l+2}$.
\end{lemma}
The proof of the above lemma is left as the content of Section \ref{sec:proofofmainlemma}. We now use the lemma to prove Theorem \ref{thm:main}.

\begin{proof}[Proof of Theorem \ref{thm:main}]
    We induct on $k$, the case of $\Pcal_0=\ZZ$ being clear. For $k\ge 1$, the quotient $\Pcal_{k+1}$ is an extension of $\Pcal_{k}$ by $H^{\otimes k+1}$ in a way comparable with equation \ref{eq:2stepSES} via the natural inclusions
    \begin{equation}
        \begin{tikzcd}
            0\rar & H^{\otimes k+1}\dar[equal]\rar& \Ical^k/\Ical^{k+2}\rar\dar[hook] & H^{\otimes k}\dar[hook]\rar & 0\\
            0\rar & H^{\otimes k+1}\rar& \Pcal_{k+1}\rar & \Pcal_k\rar & 0.
        \end{tikzcd}
    \end{equation}
    We take $\Gamma_{g,1}$-invariants and distinguish two cases by the parity of $k$. 
    
    If $k$ is odd, then the bottom row gives the exact sequence
    \begin{equation}\label{eq:P_k+1invariantSES}
    \begin{tikzcd}
        0\rar & {[H^{\otimes k+1}]^{\Sp_{2g}(\ZZ)}} \rar& {[\Pcal_{k+1}]^{\Gamma_{g,1}}}\rar & {[\Pcal_k]^{\Gamma_{g,1}}},
    \end{tikzcd}
    \end{equation}
    From induction, we know that the rightmost term is spanned by $(\zeta-1)^i+\Ical^{k+1}$ for $2i<k$, or equivalently since $k$ is odd, for $2i<k+1$; all these elements lift in $\Pcal_{k+1}$ to $(\zeta-1)^i+\Ical^{k+2}$, making the righmost map of \eqref{eq:P_k+1invariantSES} surjective, and the sequence \eqref{eq:P_k+1invariantSES} into an exact sequence of free abelian groups. It is then split, giving the desired result $[\Pcal_{k+1}]^{\Gamma_{g,1}}\cong [\Pcal_k]^{\Gamma_{g,1}}\oplus [H^{\otimes k+1}]^{\Sp_{2g}(\ZZ)}$.

    On the other hand, if $k$ is even, then the $[H^{\otimes k+1}]^{\Sp_{2g}(\ZZ)}$ vanishes giving us a square of inclusions
    \begin{equation}
            \begin{tikzcd}
             {[\Ical^k/\Ical^{k+2}]^{\Gamma_{g,1}}} \rar[hook]\dar[hook] & {[H^{\otimes k}]^{\Sp_{2g}(\ZZ)}} \dar[hook]\\ {[\Pcal_{k+1}]^{\Gamma_{g,1}}}\rar[hook] & {[\Pcal_k]^{\Gamma_{g,1}}},
        \end{tikzcd}
    \end{equation}
    and we need only determine the image of the bottom map. By induction $[\Pcal_k]^{\Gamma_{g,1}}=\ZZ\langle (\zeta-1)^i+\Ical^{k+1}:2i<k\rangle \oplus {[H^{\otimes k}]^{\Sp_{2g}(\ZZ)}}$, and the bottom maps hits all the $(\zeta-1)^i+\Ical^{k+1}$ for $2i<k$; it also hits $(\zeta-1)^{\frac{k}2}+\Ical^{k+1}=\omega^{\otimes \frac{k}2}$. The rest of $[\Pcal_k]^{\Gamma_{g,1}}$ lies in $\Ical^k$ and so in ${[\Pcal_{k+1}]^{\Gamma_{g,1}}}$, it comes from ${[\Ical^k/\Ical^{k+2}]^{\Gamma_{g,1}}}$. By the key Lemma \ref{lem:main}, we get nothing more,  thus finishing the induction.
\end{proof}

\subsection{The untruncated group ring}
For comparison the $\Pcal_k$, we now compute the invariants of the full group ring. We are able to do this over $\ZZ$. First, we require a lemma on the action of $\Gamma_{g,1}$ on $\pi$ which we could not locate in the literature; we prove it using the language of Farb--Margalit \cite{FarbMarg}.
\begin{lemma}\label{lem:finiteorbit}
    If $g\ge 1$, then every element $\g\in \pi$ that is not a power of $\zeta$ has infinite $\Gamma_{g,1}$-orbit.
\end{lemma}
\begin{proof}
    The interior of $\Sigma_{g,1}$ can be viewed as the once punctured surface $\Sigma_{g,*}$  and can be given a hyperbolic metric with the puncture $*$ forming a cusp. Under this metric, pick a representative $\widetilde{\g}$ of $\g$ as a geodesic tending to the cusp. Pick also hyperbolic representatives for the symplectic basis $a_1,b_1,\ldots, a_g,b_g$. As the complement of these $2g$ geodesics deformation retracts to an open collar of the boundary/puncture, then $\g$ must intersect one of these curves, say $\delta$, for otherwise, $\g$ would be a power of $\zeta$. Since both $\widetilde{\g}$ and $\delta$ are geodesics, this intersection is essential, and so repeated applications on $\g$ of the Dehn twist $T_\delta$ along $\delta$ only increase this intersection number. In particular,  $T_\delta^n(\g)\neq \g$ for all $n\ge 0$ giving the infinite orbit. 
\end{proof}

\begin{theorem}\label{thm:invariantsoffullring}
    If $g\ge 1$, then the invariant subring $[\ZZ\pi]^{\Gamma_{g,1}}$ is generated by $\zeta$ and $\zeta^{-1}$.
\end{theorem}
\begin{proof}
    Assume that $\sum_{i=1}^n \alpha_i\gamma_i\in \ZZ\pi,$
    where $\alpha_i\in \ZZ$ and $\gamma_i\in \pi$, is $\Gamma_{g,1}$-invariant. Then for each $\gamma_i$, its $\Gamma_{g,1}$-orbit must be a subset of $\{\gamma_1,...,\gamma_n\}$ and thus finite. By Lemma \ref{lem:finiteorbit}, then each $\gamma_i=\zeta^{k_i}$ for some $k_i\in \ZZ$.
\end{proof}
\begin{remark}\label{rem:zetamoduloIk+1}
While Theorem \ref{thm:invariantsoffullring} is essentially the inverse limit of Theorem \ref{thm:main} for $k\to \infty$, the latter could not have deduced the former because it is restricted in a stable range. We note that after truncating by powers of $\Ical$ we needed not consider the generator $\zeta^{-1}$ additionally to $\zeta$ in $\Pcal_k$ because of the identity
    \begin{equation*}
        \zeta^{-1}\equiv \sum_{i=0}^{n} (-1)^{n}(1-\zeta)^i \pmod{ I^{n+1}}.
    \end{equation*}
\end{remark}

\section{Proof of the main lemma}\label{sec:proofofmainlemma}
\subsection{Relation to Johnson} The extension \eqref{eq:2stepSES} in the key lemma is related to the Johnson homomorphism in a way we now describe. The Torelli group $\Tcal_{g,1}$ is the subgroup of $\Gamma_{g,1}$ acting trivially on $H$ and as such any element $f\in \Tcal_{g,1}$ acts trivially on the extremal terms of \eqref{eq:2stepSES}. In the special case $k=1$, the latter is the extension $H^{\otimes 2}\to \Ical/\Ical^3\to H$. Then for any $x\in H$, picking a lift $\tilde{x}\in \Ical/\Ical^3$, the quantity $\delta_f(x)=f*\tilde{x}-x$, where $f*$ denotes the action of $f$, lies in the injective image of $H^{\otimes 2}$, giving us a linear map $\delta_f:H\to H^{\otimes 2}$. The assignment $\tau:f\mapsto \delta_f$ is (the rationalisation of) the Johnson homomorphism $\Tcal_{g,1}\to \hom(H,H^{\otimes 2})$ and is a group homomorphism. 
For general $k$, a similar argument for \eqref{eq:2stepSES}, gives a group homomorphism
$$\tau^{k}:\Tcal_{g,1}\to\hom(H^{\otimes k},H^{\otimes k+1}), f\mapsto \delta^k_f.$$ 
\begin{proposition}\label{prop:deltaextendsasderivation} For any $f\in \Tcal_{g,1}$, we have
    $\delta^k_f=\sum_{i=1}^k\id_H^{\otimes i-1}\otimes \delta_f\otimes \id_H^{\otimes k-i}$. In other words, the linear map $\oplus_{k\ge 0}\delta_f^k:T[H]\to T[H]$ is a derivation of degree $1$ (but with no Koszul sign). 
\end{proposition}
\begin{proof}
    Let us write for any $y\in \Ical$, $d_f(y)=y-f*(y)$, so that $d_f(y)\in \Ical^2$ by the above argument. (Then $\delta_f$ is simply the reduction of $d_f$ modulo $\Ical^3$). Now $H^{\otimes k}$ is spanned by monomial tensors which lift in $\Ical^k$ to elements $y=(\g_1-1)\cdots (\g_k-1)$ with $\g_1,\ldots, \g_k\in \pi$. Apply $f$ to $y$ and get
    \begin{align}
        f*y&=(f*(\g_1-1))\cdots(f*(\g_k-1))\\
        &(\g_1 -1+d_f(\g_1-1))\cdots (\g_k-1+d_f(\g_1-1)
    \end{align}
    where in each of the $k$ brackets the first summand is in $\Ical$ and the second in $\Ical^2$; then modulo $\Ical^{k+2}$, the latter expression reduces to $y+\sum_{i=1}^k(\g_1-1)\cdots (\g_{i-1}-1)d_f(\g_i-1)(\g_{i+1}-1)\cdots(\g_k-1)$. This gives the result.
\end{proof}
In light of this, we shall onwards drop the $k$ from the superscript of $\delta_f$.

\subsection{The Torelli element} The only element $\f\in\Tcal_{g,1}$ we will be interested in is a boundary pair twist of genus $1$ whose $\delta_\f$ is computed in \cite{JohnsonAnAbelianQuotient} by Johnson (the reader may forget the geometric interpretation from now on). It acts on the basis $a_i,b_i, 1\le i\le g$ via 
    \begin{align*}
        \delta_\f(a_1)&=b_2\wedge a_1,\\
        \delta_\f(b_1)&=b_2\wedge b_1,\\
        \delta_\f(a_2)&=a_1\wedge b_1,
    \end{align*}
    and vanishes on the rest of the generators. Here, again, the wedge product $c_1\wedge c_2$ should be read as the commutator $c_1\otimes c_2-c_2\otimes c_1$. So for a concrete application of Proposition \ref{prop:deltaextendsasderivation} on $b_1\otimes b_3\otimes a_2\in H^{\otimes 3}$, we obtain
    \begin{align*}
        \delta_\f^3(b_1\otimes b_3\otimes a_2)=&b_2\wedge b_1\otimes b_3\otimes a_2+b_1\otimes b_3\otimes a_1\wedge b_1\\
        =&b_2\otimes b_1\otimes b_3\otimes a_2-b_1\otimes b_2\otimes b_3\otimes a_2 \\
        &+ b_1\otimes b_3\otimes a_1\otimes b_1-b_1\otimes b_3\otimes b_1\otimes a_1\in H^{\otimes 4}.
    \end{align*}
    To aleviate notation, from now on we will omit the tensor wheels and remember that the variables $a_1,b_1,\ldots, a_g,b_g$ do not commute.

\subsection{Types of monomials and the action of $\delta_\f$}\label{subsec:types} We henceforth impose the condition $2l+1\le 2g$.

In light of the previous discussion, an element $x\in H^{\otimes k}$ in the image of map \eqref{eq:invariantinjection} satisfies $\delta_f(x)=0$ for all $f\in \Tcal_{g,1}$. Then, to prove Lemma \ref{lem:main}, 
it suffices to find $f\in \Tcal_{g,1}$ which acts non-trivially on lifts of $[H^{\otimes 2l}]^{\Sp_{2g}(\ZZ)}$ except on multiples of $\omega^{\otimes l}$. We will use $f=\f$. 

Under the assumption $2g\ge 2l+1$, the $\omega_C$, with $C\in\Ccal_{2l}$, form a basis of $[H^{\otimes 2l}]^{\Sp_{2g}(\ZZ)}$. What we will need to do is distinguish the images $\delta_f(\omega_C)$ between different $C$. 

Now, to begin with each $\omega_C$ is a sum of monomials that, up to permuting the factors, are of the form $a_{i_1}b_{i_1}\cdots a_{i_l}b_{i_l}$ for $1\le i_1,\ldots, i_l\le g$. By Proposition \ref{prop:deltaextendsasderivation}, if we forget the non-commutativity of the variables for a moment, $\delta_\f$ produces, out of a monomial of the above type, new monomials which replace exactly one of $a_1$, $b_1$, $a_2$ by $b_2a_1$, $b_2b_1$, $a_1b_1$, respectively. (Notice that in each case the number of $a$-generators is preserved and the number of $b$-generators increases by $1$.) Our strategy is to find one of these new monomial which (i) survives with non-vanishing coefficient in $\delta_\f(\omega_C)$, and (ii) appears in other $\delta_\f(\omega_{C'})$ for $C'\neq C$. We shall make this concrete.

Consider monomials in the set of non-commutative variables $$\Bcal=\{a_1,...,a_g,b_1,...,b_g\}.$$
\begin{definition}[Types]  
     We will say that two monomials are \textit{of the same type} if they agree after allowing the variables to commute. This defines an equivalence relation on the set of monomials, and a \textit{type} is an equivalence class of this relation.  
\end{definition}
We declare the family of types $AB_{2l}$ to contain all monomials  of type $$a_{i_1}b_{i_1}a_{i_2}b_{i_2}...a_{i_l}b_{i_l}$$ where $i_1,...,i_l\in \{1,...,g\}$. Then each $\omega_C$ for $C\in\Ccal_{2l}$ is of type $AB_{2l}$. We further introduce the types 
\begin{align*}
    X_1: \hspace{6pt} &\underline{a_1b_1}a_3b_3...a_{l+1}b_{l+1},\\
    X_2: \hspace{6pt}  &\underline{a_2b_2}a_3b_3...a_{l+1}b_{l+1},\\
    Y: \hspace{6pt} &\underline{a_1b_1b_2}a_3b_3...a_{l+1}b_{l+1},
\end{align*}
(where we have underlined the distinguishing factors) 
and we will call type $X$ the union of types $X_1$ and $X_2$. The motivation is that for a monomial $\mfrak$ of type $X$, which is in the family $AB_{2l}$, $\delta_\f(\mfrak)$  is of type $Y$.  The next technical lemma summarises the more specific way $\delta_\f$ operates on monomials of type $X$.

\begin{lemma}\label{lem:Tfm}
\begin{enumerate}
    \item If $\mfrak$ is a monomial of type $X_1$, resp. $X_2$, then $\delta_\f(\mfrak)$ is a linear combination of four, resp. two, distinct monomials of type $Y$ with coefficients $\pm 1$.
    \item Let $\mfrak_i$ and $\mfrak'_j$ be monomials of type $X_i$ and $X_j$ respectively, where $i,j\in \{1,2\}$. Then $\delta_\f(\mfrak_i)$ and $\delta_\f(\mfrak'_j)$ share a common monomial summand if and only if either (a)  $i=j$ and $\mfrak_i=\mfrak'_j$, or (b) $i\neq j$ and there exist monomials $\vfrak,\wfrak$ in the set of generators $\Bcal-\{a_1,b_1,a_2,b_2\}$ such that $$\mfrak_i=\vfrak x_iy_i\wfrak$$ and $$\mfrak'_j=\vfrak x'_jy_j'\wfrak,$$
    where $\{x_i,y_i\}=\{a_i,b_i\}$ and $\{x'_j,y'_j\}=\{a_j,b_j\}$.
\end{enumerate}
\end{lemma}
\begin{proof}
    For each of $i=1,2$, we split the type $X_i$ of monomials into two ``subtypes'' $A_i$ and $B_i$ depending on which of their unique $a_i$ and $b_i$ factors comes first. Specifically, a monomial of each of these subtypes can be written as
    \begin{align*}
        A_1: \hspace{6pt} &\mfrak_1=\vfrak a_1\ufrak b_1\wfrak, \\
        B_1: \hspace{6pt} &\mfrak_1=\vfrak b_1\ufrak a_1\wfrak,  \\
        A_2: \hspace{6pt} &\mfrak_2=\vfrak a_2\ufrak b_2\wfrak,  \\
        B_2: \hspace{6pt} &\mfrak_2=\vfrak b_2\ufrak a_2\wfrak 
    \end{align*}
    with $\ufrak ,\vfrak,\wfrak$ monomials in $\Bcal-\{a_1,b_1,a_2,b_2\}$. Using that $\delta_\f$ is a derivation (Proposition \ref{prop:deltaextendsasderivation}), we evaluate
    \begin{align}
        A_1: \hspace{6pt} \delta_\f(\mfrak_1)=&\vfrak (a_1\wedge b_2)\ufrak b_1\wfrak +\vfrak a_1\ufrak (b_1\wedge b_2)\wfrak, \label{eq:deltafA1} \\
        B_1: \hspace{6pt} \delta_\f(\mfrak_1)=&\vfrak (b_1\wedge b_2)\ufrak a_1\wfrak +\vfrak b_1\ufrak (a_1\wedge b_2)\wfrak, \label{eq:deltafB1}\\
        A_2: \hspace{6pt} \delta_\f(\mfrak_2)=&\vfrak (a_1\wedge b_1)\ufrak b_2\wfrak, \label{eq:deltafA2}\\
        B_2: \hspace{6pt} \delta_\f(\mfrak_2)=&\vfrak b_2\ufrak (a_1\wedge b_1)\wfrak. 
    \end{align}
    Recalling that $a\wedge b=ab-ba$, the first two lines produce $4$ monomial summands each with sign $\pm 1$, while the last two lines $2$ each. It is clear by looking at the order of the factors $a_1,b_1,a_2,b_2$, that the monomials in each line are distinct. This proves assertion (1).

    In sequel, we study under what conditions $\delta_\f(\mfrak_i)$ and $\delta_\f(\mfrak'_i)$ share common monomial summands. We distinguish cases depending on whether $\mfrak_i$ and $\mfrak_i'$ are of the same type and/or subtype.

    \emph{Case 1: $\mfrak_i$ and $\mfrak'_i$ of same subtype.} We assume this subtype is $A_1$; the same argument applies to the other three subtypes. Suppose $\delta_\f(\mfrak_1)$, $\delta_\f(\mfrak'_1)$ share a common monomial summand, and write $\mfrak_1=\vfrak a_1\ufrak b_1\wfrak$ and $\mfrak'_1=\vfrak'a_1\ufrak'b_1\wfrak'$. Now the four monomial summands from \eqref{eq:deltafA1} are uniquely distinguished by the order of appearance of $a_1,b_1,b_2$. So if, say, the summand $\vfrak a_1b_2\ufrak b_1\wfrak$ of $\delta_\f(\mfrak_1)$ appears as a summand of $\delta_\f(\mfrak'_1)$, then this is the summand $\vfrak'a_1b_2\ufrak'b_1\wfrak'$. By the running assumption that $\vfrak ,\vfrak',\ufrak ,\ufrak',\wfrak,\wfrak'$ have no $a_1,b_1,a_2,b_2$ as factor, the equality $\vfrak a_1b_2\ufrak b_1\wfrak=\vfrak'a_1b_2\ufrak'b_1\wfrak'$ implies $\vfrak =\vfrak'$, $\ufrak =\ufrak'$ and $\wfrak=\wfrak'$, yielding $\mfrak_1=\mfrak'_1$. The same conclusion follows if we start from any of the other three summands of $\delta_\f(\mfrak_1)$.

    \emph{Case 2: $\mfrak_i$ and $\mfrak'_i$ of same type but different subtype.} Assume $i=1$, and $\mfrak_1,\mfrak_1'$ are of subtypes $A_1,B_1$, respectively. Then, all four monomial summands of $\delta_\f(\mfrak_1)$ from \eqref{eq:deltafA1} have the factor $a_1$ appearing before $b_1$, whereas all monomials of  $\delta_\f(\mfrak_1')$ from \eqref{eq:deltafB1} have $b_1$ before $a_1$. We then find no common monomials. The case $i=2$ is analogous but simpler.

    \emph{Case 3: $\mfrak_i$ and $\mfrak'_i$ of different type.}
    We assume $\mfrak_1=\vfrak a_1\ufrak b_2\wfrak$ is of type $A_1$ and $\mfrak'_2=\vfrak a_2\ufrak b_2\wfrak$ is of type $A_2$. Then an inspection of \eqref{eq:deltafA1} and \eqref{eq:deltafA2} finds only one pair of monomials in $\delta_\f(\mfrak_1)$ and $\delta_\f(\mfrak'_2)$ having the factors  $a_1,b_2,b_1$ in the same order. These are, respectively,
    \begin{equation}\label{eq:commonmonomials}
        \vfrak a_1\ufrak b_1b_2\wfrak \text{ and } \vfrak 'a_1b_1\ufrak'b_2\wfrak'
    \end{equation}
    and they are equal if and only if $\vfrak =\vfrak'$, $\wfrak=\wfrak'$ and $\ufrak $ and $\ufrak'$ are both the empty monomial. In other words, under the current assumption, $\delta_\f(\mfrak_1)$ and $\delta_\f(\mfrak'_2)$ share a common monomial if and only $\mfrak_1=\vfrak a_1b_1\wfrak$ and $\mfrak_2=\vfrak a_2b_2\wfrak$, as claimed. The other three cases are done similarly, with the common monomial in place of \eqref{eq:commonmonomials}  shown in Table \ref{tab:commonmonomial}. 
\end{proof}

    \begin{table}
        \centering
        \begin{tabular}{c|c|c}
          & $A_1$ & $B_1$ \\ \hline
        $A_2$ & $\vfrak a_1b_1b_2\wfrak$& $\vfrak b_1a_1b_2\wfrak$\\ \hline
        $B_2$ & $\vfrak b_2a_1b_1\wfrak$ & $\vfrak b_2b_1a_1\wfrak$ \\ 
    \end{tabular}
        \caption{The unique common monomial between $\delta_\f(\mfrak_1)$ and $\delta_\f(\mfrak'_2)$ depending on their subtypes.}
        \label{tab:commonmonomial}
    \end{table}

\subsection{Action on chord-diagrams}
We now analyse the action of $\delta_\f$ on the $\omega_C$ for $C\in \Ccal_{2l}$. In expressing $\omega_C$ as a polynomial in $\Bcal$, we find that it always contains monomials of type $X_j$ for each $j=1,2$: for example, the monomial where $a_j,b_j$ are placed in the tensor slots $p_1,q_1$, respectively, and, for $i\ge 2$, the factors $a_{i+1},b_{i+1}$ are placed in the slots $p_i,q_i$, respectively. Conversely, if we are given a monomial of type $X_j$, with $j=1,2$, there is a unique $\omega_C$ which contains it as a summand: simply ``join with a chord'' each pair of tensor slots $\{p_i,q_i\}$ where appear an $a$ factor and a $b$ factor with the same index; this uniquely defines a chord diagram. With this idea, we can prove
\begin{proposition}[No cancellations]\label{prop:nocancellations}
    Let $C,C'\in \Ccal_{2l}$, and $\mfrak, \mfrak'$ be type $X$ monomial summands of $\omega_C, \omega_{C'}$, respectively. Then $\delta_\f(\mfrak)$ and $\delta_\f(\mfrak')$ share a common monomial summand (of type $Y$) only if $C=C'$.
\end{proposition}
\begin{proof}
    By Lemma \ref{lem:Tfm}, if $\delta_\f(\mfrak)$ and $\delta_\f(\mfrak')$ share a polynomial summand then either $\mfrak=\mfrak'$ or $\mfrak=vx_iy_iw$ and $\mfrak'=vx'_{i'}y'_{i'}w,$
    where $\{x_i,y_i\}=\{a_i,b_i\}$ and $\{x'_{i'},y'_{i'}\}=\{a_{i'},b_{i'}\}$. In either case, the recipe described above reads off the same chord diagram, so $C=C'$.
\end{proof}

\begin{proposition}[Non-triviality]\label{prop:nontriviality}
    If $C\neq C_0$, then the expression of $\delta_\f(\omega_C)$ as a polynomial in the set $\Bcal$ has monomials of type $Y$ with non-zero coeffiecients. If $C=C_0$, then $\delta_\f(\omega_{C_0})$ vanishes in $H^{\otimes 2k+1}$.
\end{proposition}
\begin{proof}
    Assuming $C\neq C_0$, then $C$ has a non-consecutive chord, i.e. a chord $(p_i,q_i)$ with $q_i\neq p_i+1$. Then $\omega_C$ has a monomial summand 
    $\mfrak$ of type $X_1$ that has $a_1,b_1$ in the tensor slots $p_i,q_i$, respectively. By part (2) of Lemma \ref{lem:Tfm}, there is no other monomial $\mfrak'$ in $\omega_C$ so that $\delta_\f(\mfrak')$ and $\delta_\f(\mfrak)$ share common monomials. Thus the type $Y$ summands of $\delta_\f(\mfrak)$ all survive intact in the summation $\delta_\f(\omega_C)$ with coefficients $\pm 1$. If on the other hand $C=C_0$, then $\omega_{C_0}=\omega^{\otimes k}$. An application of Proposition \ref{prop:deltaextendsasderivation} shows $\delta_\f(\omega)=0$ (in fact this is true for all Torelli elements), and an extension of the derivation rule concludes $\delta_\f(\omega^{\otimes k})=0$.
\end{proof}

The last two propositions combine to yield  
\begin{proposition}\label{thm:LI}
    If $2l+1\le 2g$, then the set $$\{\delta_\f(\omega_C): C\in \Ccalk-\{C_0\}\}$$
    is linearly independent in $H^{\otimes 2l+1}$ and $\delta_\f(\omega_{C_0})=0$.
\end{proposition}

We can now conclude our proof.
\begin{proof}[Proof of Lemma \ref{lem:main}]
    If $v\in [\Ical^{2l}/\Ical^{2l+2}]^{\Gamma_{g,1}}$, then its image $[v]\in [H^{\otimes 2l}]^{\Sp_{2g}(\ZZ)}$ is, by Theorem \ref{thm:chorddiagramsinvariants}, a linear combination  $[v]=\sum_{c\in \Ccalk}\alpha_C\omega_C$ where $\alpha_C\in \ZZ$. By the definition of $\delta_\phi$ and the invariance of $v$, we must have $\delta_\f([v])=0$, and so $$\sum_{C\in \Ccal_{2l}}\alpha_C\delta_\f(\omega_C)=0.$$ From the vanishing of $\delta_\f(\omega_{C_0})$ and the linear independence of the rest of the $\delta_\f(\omega_C)$ (Proposition \ref{thm:LI}) it follows that $\alpha_C=0$ for $C\neq C_0$ and so $[v]=\alpha \omega^{\otimes l}$ for some $\alpha\in \ZZ$. From Example \ref{ex:zetaISomega}, $v=(\zeta-1)^l+\Ical^{2l+2}$ maps to $\omega^{\otimes l}$, so the image of the injective map \eqref{eq:invariantinjection} is spanned by $\omega^{\otimes l}$ and the domain by $(\zeta-1)^l+\Ical^{2l+2}$.
\end{proof}

\begin{remark}\label{rem:unstablerange}
We used the assumption $2l+1\le 2g$ to have enough generators $a_i,b_i$ to be able to detect from a monomial in $\delta_\f(\omega_C)$ the chord diagram $C$. A more economical argument could emerge by (i) using a basis of $[H^{\otimes 2l}]^{\Sp_{2g}(\ZZ)}$ for smaller $g$ given as a subset of the chord provided by Mihailovs \cite{mihailovs1998symplectictensorinvariantswave}, and (ii) by using the action of more Torelli elements.
\end{remark}

\bibliographystyle{amsalpha}

\bibliography{biblio.bib}
\end{document}